\documentclass[letterpaper, 10 pt, conference]{ieeeconf}
\IEEEoverridecommandlockouts
\usepackage{amsmath}
\usepackage{amssymb}
\usepackage{graphicx}

\usepackage{mathrsfs}

\newtheorem{thm}{\protect\theoremname}

\newtheorem{prop}[thm]{\protect\propositionname}
\newtheorem{rem}[thm]{\protect\remarkname}

\providecommand{\definitionname}{Definition}
\providecommand{\propositionname}{Proposition}
\providecommand{\remarkname}{Remark}
\providecommand{\theoremname}{Theorem}
\providecommand{\lemmaname}{Lemma}

\begin{document}

\title{Semidefinite Relaxations for Stochastic Optimal Control Policies}

\author{Matanya B. Horowitz, Joel W. Burdick
\thanks{Matanya Horowitz and Joel Burdick are with the 
	Department of Control and Dynamical Systems,
        Caltech, 1200 E California Blvd., Pasadena, CA.
        The corresponding author is available at {\tt\small mhorowit@caltech.edu}}%
\thanks{Matanya Horowitz is supported by a NSF Graduate Research Fellowship.%
}
}

\maketitle
\thispagestyle{empty}
\pagestyle{empty}

\begin{abstract}
Recent results in the study of the Hamilton Jacobi Bellman (HJB) equation have led to the discovery of a formulation of the value function as a linear Partial Differential Equation (PDE) for stochastic nonlinear systems with a mild constraint on their disturbances. This has yielded promising directions for research in the planning and control of nonlinear systems. This work proposes a new method obtaining approximate solutions to these linear stochastic optimal control (SOC) problems. A candidate polynomial with variable coefficients is proposed as the solution to the SOC problem. A Sum of Squares (SOS) relaxation is then taken to the partial differential constraints, leading to a hierarchy of semidefinite relaxations with improving sub-optimality gap. The resulting approximate solutions are shown to be guaranteed over- and under-approximations for the optimal value function.
\end{abstract}

\section{Introduction}

As robots and autonomous systems are fielded in increasing complex situations, the ability to move safely in the presence of uncertain actuation and sensing, as well as dynamically changing and uncertain environments, becomes ever more important. Practically useful control and planning methods must also be rapidly computable, and should incorporate optimality criteria when possible.

Sampling based planners such as PRMs and RRTs \cite{LaValle:tb} have become popular for such problem because they are adaptable to a variety of problems, and often have rapidly computable solutions in higher dimensional problems. How- ever, these approaches typically rely on an abstraction of the state space that eliminates considerations such as stochasticity and dynamics. Secondary issues such as control effort and movement efficiency cannot also be readily incorporated in these main stream approaches.

Stochastic optimal control (SOC) provides an alternative framework, allowing for various important details of the problem to be directly incorporated into the motion planning formulation and solution. Many such SOC problems are discretized, resulting in Markov Decision Problems that can be solved through methods such as Value Iteration \cite{Bertsekas:1995uq}. In robotic applications, such discretizations become prohibitively difficult to solve due to the curse of dimensionality associated with robotic systems of even moderate complexity.

Recently it has been discovered that the (typically non- linear) Hamilton Jacobi Bellman (HJB) equation of optimal control may be transformed to a linear PDE given several mild assumptions \cite{Kappen:2005kb, Todorov:2009wja}.  This approach might lead to significant computational gains, allowing for practical applications of SOC. To date, this class of linearly solvable HJB problems has been solved through sampling methods suggested by the Feynman-Kac Lemma \cite{Theodorou:2011uz}.

This paper presents a novel alternative method to solve such problems using polynomial optimization and semidefinite programming.  Using {\em sum of squares} (SOS) techniques, \cite{Parrilo:2003ul}, we construct an approximate value function that satisfies the linearly solvable HJB equation. This allows for optimal control problems, including those typically found in robotic motion planning, to be computed quickly, with globally optimal solutions. In contrast to dynamic program- ming approaches, no discretization is required, postponing the curse of dimensionality and eliminating a potential source of approximation error. Moreover, our formulation leads directly to gap theorems, or bounds, on the approximation error.

{\bf Related Work.}  The study of linearly solvable SOC problems has developed along two lines of investigation. One is that of Linear MDPs \cite{Todorov:2009wja}, in which an MDP may be solved as a linear set of equations given several assumptions. By taking the continuous limit of the discretization, a linear PDE is obtained. In another line of work begun by Kappen
\cite{Kappen:2005kb} the same linear PDE has been found through a particular transformation of the HJB. Existing approaches for solving the resulting linear HJB PDE have focused on the use of the Feynman-Kac Lemma, which relates the solution to the linear PDE to the diffusion of a stochastic process. This approach has been developed by Theodorou et al. \cite{Theodorou:2010vd} into the Path Integral framework in use with Dynamic Motion Primitives. Therein, sampling is augmented with the use of suboptimal policies, producing better estimates of the dynamics when executing an optimal policy. The resulting samples trajectories can then be used to in turn improve the policy, and then the process is iterated. These results have been developed in a number of compelling directions \cite{Theodorou:2012wx,Theodorou:2010vk,Dvijotham:2011tv}.

The sampling-based approach developed through the Fenyman-Kac Lemma is an alternative to the approach presented here, with several potential advantages and disadvantages. Among these, sampling-based approaches such as that of Theodorou, may be more attractive in high dimensional state spaces.  However, the approach presented in this paper may be quite rapid in practice, produces a global solution with no need for iteration, and the scalability of the process is an open question that this paper only begins to investigate.  While  it is beyond the
scope of this paper, the framework presented below may find applications in the method of Control Lyapunov Functions \cite{Primbs:1999we} and Receding Horizon Control \cite{jadbabaie2001unconstrained}.

\section{Background}

Using techniques developed for polynomial optimization, we develop a method to obtain a universal approximation for the value function with the best error as measured in the pointwise norm. We will review the development of the linear optimal control PDE, along with the necessary tools of polynomial optimization.

\subsection{A Linear Hamilton-Jacobi-Bellman (HJB) Equation}

A common construction in the optimization literature is the value function, which captures the ``cost-to-go'' from a given state.  If such a quantity is known, the optimal action may be chosen as that which follows the gradient of the value, bringing the agent into the states with highest value over the remaining time horizon. The construction of the value function $V(x)$ presented here follows the development in \cite{EvangelosIFAC11}.

We focus on system with state $x_t\in\mathbb{R}^{n}$ at time $t$, control input
$u_t\in\mathbb{R}^{m}$, and dynamics that evolve according to the equation
  \begin{equation} \label{eq:stochastic-dynamics}
    dx_t = \left(f(x_t)+G(x_t)u_t\right)dt+B(x_t) \mathcal{L}\, d\omega_t
  \end{equation}
where the expressions $f(x),G(x),B(x)$ are assumed to be polynomial functions of the state variables $x$, and $\omega$ is a Brownian motion with (i.e., a stochastic process such that $\omega_t$ has independent increments with $\omega_t-\omega_s\sim N(0,t-s)$ for $N(\mu,\sigma^2)$ a
normal distribution). The matrix $\mathcal{L}$ is constant.  The system has costs $r_{t}$ accrued at time $t$ according to
  \begin{equation}
    r(x_{t},u_{t})=q(x_t)+\frac{1}{2}u_{t}^{T}Ru_{t}\label{eq:cost}
   \end{equation}
where $q(x)$ is a state dependent cost and the control effort enters quadratically.  We require $q(x)\ge 0$ for all $x$ in the problem domain (which is more carefully defined below).

The goal is to minimize the expected cost of the following functional,
  \begin{equation} \label{eq:cost-functional}
       J(x,u)=\phi_{T}(x_{T})+\int_{0}^{T}r(x_{t},u_{t}) dt
  \end{equation}
where $\phi_{T}$ represents a state-dependent terminal cost. The solution to this minimization is
known as the {\bf value function}, where, beginning from an initial point $x_{t}$ at time $t$
  \begin{equation} \label{eq:value-def}
   V\left(x_{t}\right)=\min_{u_{t:T}}\mathbb{E}\left[J\left(x_{t}\right)\right] 
  \end{equation}
with $u_{t:T}$ being short-hand notation for $u(t)$, $t\ \in\ [t,T]$.

The Hamilton-Jacobi-Bellman equation associated with this problem, arising from Dynamic Programming
arguments \cite{Fleming:2006tl}, is found to be
  \begin{equation} \label{eq:original-hjb}
       -\partial_{t}V = \min_{u}\left(r+\left(\nabla_{x}V\right)^{T}f 
        +\frac{1}{2}Tr\left(\left(\nabla_{xx}V\right)G\Sigma_{\epsilon}G^{T}\right)\right)
  \end{equation}
where we define $\Sigma_{\epsilon}=\mathcal{L}\mathcal{L}^{T}$. As the control effort enters
quadratically into the cost function it is a simple matter to solve for it analytically in
(\ref{eq:cost}):
  \[ u^{*}=-R^{-1}G^{T}V_{x}. \]
The minimal control, $u^{*}$, may then be substituted into (\ref{eq:original-hjb}) to yield the
following nonlinear, second order partial differential equation (PDE):
  \begin{eqnarray*}  \label{eq:hjb-pde-value}
     -\partial_{t}V & = & q+\left(\nabla_{x}V\right)^{T}f-\frac{1}{2}
           \left(\nabla_{x}V\right)^{T}GR^{-1}G^{T}\left(\nabla_{x}V\right)\\
                    & + & \frac{1}{2}Tr\left(\left(\nabla_{xx}V\right)B\Sigma_{\epsilon}B^{T}\right).
\end{eqnarray*}
The difficulty of solving this PDE is what usually prevents practitioners of optimal control from attempting to solve for the value function directly. However, it has recently been found \cite{EvangelosIFAC11,Todorov:2009wja,Kappen:2005bn} that with the assumption that there exists a $\lambda\in\mathbb{R}$ and a control penalty cost $R\in\mathbb{R}^{n\times n}$ satisfying this equation
  \begin{equation} \label{eq:noise-assumption}
    \lambda G(x)R^{-1}G(x)^{T}=B(x)\Sigma_{\epsilon}B(x)^{T}\triangleq\Sigma_{t}
   \end{equation}
and using the logarithmic transformation 
  \begin{equation} \label{eq:log-transform}
      V=-\lambda\log\Psi
  \end{equation}
it is possible, after substitution and simplification, to obtain the following {\em linear}
PDE from Equation (\ref{eq:hjb-pde-value}).
  \begin{equation} \label{eq:hjb-pde}
   -\partial_{t}\Psi = -\frac{1}{\lambda}q\Psi+f^{T}(\nabla_{x}\Psi)
      +\frac{1}{2}Tr\left(\left(\nabla_{xx}\Psi\right)\Sigma_{t}\right).
  \end{equation}
This transformation of the value function, which we call here the \emph{desirability}
\cite{Todorov:2009wja}, provides an additional, computationally appealing method through which to
calculate the value function.

Similar arguments may be made to develop value functions in an additional problems of
interest. These are listed in Table \ref{tab:PDE-types}. For brevity, an expression common to the
desirability equations is defined
  \begin{equation}\label{eq:L-psi}
     L(\Psi):=f^{T}(\nabla_{x}\Psi)+\frac{1}{2}Tr\left(\left(\nabla_{xx}\Psi\right)\Sigma_{t}\right)
  \end{equation}

\begin{rem}The condition (\ref{eq:noise-assumption})
can roughly be interpreted as a controllability-type condition: the system must have sufficient
control to span (or counterbalance) the effects of input noise on the system dynamics. A degree of
designer input is also given up, as the constraint restricts the design of the control penalty $R$,
requiring that control effort be highly penalized in subspaces with little noise, and lightly
penalized in those with high noise. Additional discussion is given in \cite{Todorov:2009wja}.
\end{rem}

\begin{table}
\begin{centering}
\begin{tabular}{|c|c|c|}
\hline 
 & Cost Functional & Desirability PDE\tabularnewline
\hline 
Finite & $\phi_{T}(x_{T})+\int_{0}^{T}r(x_{t},u_{t})$ & $\frac{1}{\lambda} q \Psi -\frac{\partial\Psi}{\partial t}=L(\Psi)$\tabularnewline
\hline 
First-Exit & $\phi_{T_{*}}(x_{T_{*}})+\int_{0}^{T_{*}}r(x_{t},u_{t})$ & $\frac{1}{\lambda} q \Psi=L(\Psi)$\tabularnewline
\hline 
Average & $\lim_{T\to\infty}\frac{1}{T}\mathbb{E}\left[\int_{0}^{T}r(x_{t},u_{t})\right]$ & $\frac{1}{\lambda} q \Psi -c\Psi=L(\Psi)$\tabularnewline
\hline 
\end{tabular}
\par\end{centering}

\caption{Constraints on the desirability function arising in a number of common
optimal control problems \cite{Todorov:2009wja}.
\label{tab:PDE-types}}
\end{table}

\subsection{Sum of Squares Programming}

We provide a brief review on Sum of Squares (SOS) programming, with additional technical details
available in \cite{Parrilo:2003fh}. These tools will be key in the development of approximate
solutions to \eqref{eq:hjb-pde}. In brief, \eqref{eq:hjb-pde} specifies a set of \emph{partial
differential} equality constraints that the optimal solution must satisfy. To develop instead a
close approximation these equality constraints may be relaxed to inequalities. The optimization
problem is then to find the best approximate solution that lies in the set of polynomials that
satisfy these inequality constraints, known as a semialgebraic set. SOS provides a method to perform
optimization over such a set.

Formally, a \emph{semialgebraic set} is a subset of $\mathbb{R}^{n}$ that is specified by a finite
number of polynomial equations and inequalities.  An example is
  \[ \left\{ \left(x_{1},x_{2}\right)\in\mathbb{R}^{2}\mid x_{1}^{2} 
       + x_{2}^{2}\le1,x_{1}^{3}-x_{2}\le0\right\}. \]
Such a set is not necessarily convex, and testing membership in the set is intractable in general
\cite{Parrilo:2003fh}. As we will see, however, there exists a class of semialgebraic sets that are
in fact semidefinite-representable. Key to this development will be first the ability to test for
non-negativity of a polynomial.

A multivariate polynomial $f(x)$ is a \emph{sum of squares} (SOS) if there exist polynomials
$f_{1}(x),\ldots,f_{m}(x)$ such that
  \[ f(x)=\sum_{i=1}^{m}f_{i}^{2}(x). \]
A seemingly unremarkable observation is that a sum of squares is always positive. Thus, a sufficient
condition for non-negativity of a polynomial is that the polynomial is SOS. Perhaps less obvious is
that membership in the set of SOS polynomials may be tested as a convex problem. We denote the
function $f(x)$ being SOS as $f(x)\in\Sigma(x)$.

\begin{thm} \label{thm:sos-test-sdp}
(\cite{Parrilo:2003fh}) Given a finite set of polynomials $\left\{ f_{i}\right\}
_{i=0}^{m}\in\mathbb{R}^{n}$ the existence of $\left\{ a_{i}\right\} _{i=1}^{m}\in\mathbb{R}$ such
that
  \[ f_{0}+\sum_{i=1}^{m}a_{i}f_{i}\in\Sigma(x) \]
is a semidefinite programming feasibility problem. 
\end{thm}

Thus, while the problem of testing non-negativity of a polynomial is intractable in general, by constraining the feasible set to SOS the problem becomes tractable. The converse question, is a non-negative polynomial necessarily a sum of squares, is unfortunately false, indicating that this test is conservative \cite{Parrilo:2003fh}. Nonetheless, SOS feasibility will be sufficiently
powerful for our purposes.

\subsubsection{The Positivstellensatz}

Using SOS theory, it is possible to determine whether a particular polynomial, possibly
parameterized, is a sum of squares. The next step is to determine how to combine multiple polynomial
inequalities, i.e. semialgebraic sets of the form
  \[ P=\left\{ x\in\mathbb{R}^{n}\mid f_{i}(x)\ge0\mbox{ for all }i=1,\ldots,m\right\}  \]
for polynomial functions $f_{i}(x)$ where $x\in\mathbb{R}^{n}$.  The answer is given by Stengle's
\emph{Positivstellensatz}.

\begin{thm} (Positivstellensatz \cite{Stengle:1974ix}) The set
   \[ X=\left\{ x\mid f_{i}(x)\ge0,h_{j}(x)=0\text{ for all }i=1,\ldots,m,\, j=1,\ldots,p\right\} \]
is empty if and only if there exists $t_{i}\in\mathbb{R}[x]$, $s_{i},r_{ij},\ldots\in\Sigma$
such that
   \[ -1=s_{0}+\sum_{i}h_{i}t_{i}+\sum_{i}s_{i}f_{i}+\sum_{i\neq j}r_{ij}f_{i}f_{j}+\cdots \]
\label{thm:positivestellensatz}
\end{thm}

Although this theorem is presented in terms of feasibility, it is easily adapted for the purposes of
optimization. Given the problem
  \begin{flalign*}
    \min\,\, & f_{0}(x)\\
    \text{subject to}\,\, & f_{i}(x)\le0\,\,\,\,\forall i\in1,\ldots,k
   \end{flalign*}
a slack factor $\gamma$ may be introduced to frame the equivalent infeasibility problem
  \begin{flalign*}
      \max\,\, & \gamma\\
       s.t.\,\, & \left.\begin{array}{l}
       f_{0}(x)\le\gamma\\
       f_{i}(x)\le0\,\,\,\,\forall i\in1,\ldots,k
       \end{array}\right\} \text{\emph{infeasible}}
  \end{flalign*}
which is in a directly applicable form for the \emph{Positivestellensatz}. 

A sufficient condition for infeasibility may be created by limiting the inclusion of some of the
multipliers, e.g. setting some to zero such as $r_{ij}$ or $r_{ijk}$. Alternatively, it is possible
to limit the degree of the multipliers $h_{i},s_{i},r_{ij}$. In the search for infeasibility we may
therefore begin with a limited polynomial degree, increasing the degree if additional precision is
required. This creates a \emph{hierarchy} of semidefinite relaxations of increasing complexity but
also with a decrease in the suboptimality of the solution. This construction is known more broadly
as a Theta Body relaxation \cite{gouveia2010theta}.

\section{Sum-of-Squares Relaxation of the HJB PDE}

Sum of squares programming has found many uses in combinatorial optimization, control theory, and
other applications. We now expand its use to include finding approximate solutions to the value
function of the stochastic optimal control problem.

Obtaining solutions to linear PDEs is far from trivial. However, we propose to first approximate the
desirability solution to the linear HJB PDE as a polynomial. While the value function may in fact be
discontinuous, we make the modeling assumption that it may be approximated to a sufficiently high
accuracy given a polynomial of sufficient degree. Furthermore, although the solution to the HJB is
discontinuous in some locations, in many continuous domains, such as many robotics and control
problems of interest, it will remain continuous over large portions of the domain. Historically,
difficulties with the discontinuities present in HJB equations have led to the development of
viscosity solutions \cite{Fleming:2006tl}, in effect placing a smoothness requirement on the
solution.

We proceed with the finite horizon problem, but similar steps apply to all the problems listed in Table \ref{tab:PDE-types}. We make the assumption that the control problem occurs only on a compact domain $\mathbb{S}$ that is representable as a semialgebraic set, as is its boundary $\partial\mathbb{S}$. 

The equality constraint of \eqref{eq:hjb-pde} may be relaxed, yielding the following constraints that
are necessary for an over-approximation of the desirability function
  \begin{equation}  \label{eq:over-approximation-relax}
     \frac{1}{\lambda}q\Psi \le \partial_{t}\Psi+f^{T}(\nabla_{x}\Psi) 
        + \frac{1}{2} Tr\left(\left(\nabla_{xx}\Psi\right)\Sigma_{t}\right) 
\end{equation}
Hereafter, we will indicate solutions to the above inequality as $\Psi$, and exact solutions to
\eqref{eq:hjb-pde} as $\Psi^*$, the optimal desirability function. To obtain the best such approximation $\Psi$ for a given
polynomial order, the pointwise error of the approximation may be minimized in the optimization problem
\begin{flalign*}
\min\,\, & \gamma\\
s.t.\,\, & \gamma-\left(\frac{1}{\lambda} q \Psi - \partial_t \Psi -L(\Psi)\right)\ge0
\end{flalign*} 
for $x\in\mathbb{S}$. The boundary conditions of (\ref{eq:hjb-pde}) correspond to the exit conditions of the optimal control problem. In all problems this may correspond to colliding with an obstacle or goal region, and in the finite horizon problem there is the added boundary condition of the terminal cost at $t=T$. These final costs must then be transformed according to (\ref{eq:log-transform}), producing the added constraint
    \[ \Psi\mid_{\partial\mathbb{S}}=e^{-\frac{\phi_T(x_T)}{\lambda}} \]
where $\phi_T(x_T)$ is the terminal cost from \eqref{eq:cost-functional}. This constraint may be
also be relaxed as an inequality. The complete optimization problem is then
\begin{alignat}{2}
\min \quad & \gamma                                                    & \label{eq:sos_pde} \\
s.t.    \quad& \frac{1}{\lambda}q\Psi \le \partial_{t}\Psi+L(\Psi)        & x\in\mathbb{S} \nonumber\\
         & \gamma                          \ge\frac{1}{\lambda}q \Psi- \partial_t \Psi -L(\Psi) & x\in\mathbb{S}  \nonumber\\
         & \Psi                                 \le e^{-\frac{\phi_T(x)}{\lambda}}                             & x\in\partial\mathbb{S} \nonumber
\end{alignat}
As the inequalities are defined over polynomials, this optimization is defined over a semialgebraic
set. This may be made tractible as follows.

\begin{prop}\label{prop:sos-opt}
The optimization problem \eqref{eq:sos_pde} where inequality constraints are relaxed to SOS
membership may be solved as a semidefinite optimization program.
\end{prop}
\begin{proof}
Let us propose a candidate solution to the optimization $\Psi$, a polynomial of fixed degree $n$,
denoted $\Psi_n$. Each of the inequality constraints are non-negativity constraints over a
polynomial and are therefore a semialgebraic set. The full set of constraints is an intersection of
semialgebraic sets and therefore also a semialgebraic set. When the inequalities in this set are
relaxed as SOS constraints, membership in the constraint set may be tested as a semidefinite program
by Theorem \ref{thm:sos-test-sdp}. The optimization over this set is then enabled by Theorem
\ref{thm:positivestellensatz}.
\end{proof}

Furthermore, one can in fact guarantee the exact and polynomial approximate desirability functions have a bounded relationship.
\begin{thm} \label{prop:Psi_bound}
Given a solution $\left\{ \Psi,\gamma\right\} $ to \eqref{eq:sos_pde}, and if $\Psi^*$ is the
solution to \eqref{eq:hjb-pde}, then $\Psi(x) \le \Psi^{*}(x)$ for all $x\in\mathbb{S}$.
\end{thm}
\begin{proof}
Consider the first-exit case for simplicity, and define the error between approximation $\Psi$ and the optimal desirability $\Psi^*$, $e=\Psi - \Psi^*$. Then, as all operators are linear,
\begin{align*}
\frac{1}{\lambda} q e & = \frac{1}{\lambda} q \left( \Psi - \Psi^* \right) \\
 & = \frac{1}{\lambda} q \Psi - L(\Psi^*) \\
 & \le L(\Psi) - L(\Psi^*) \\
 & \le L(e)
\end{align*}

Defining the augmented operator $P(e):= L(e) - \frac{1}{\lambda} q e$ then $P$ is an elliptic operator and by the weak maximum principle for elliptic operators \cite{max_principles}
\begin{equation}
\sup_\mathbb{S} e \le \sup_{\partial \mathbb{S}} e^+
\end{equation}
where $e^+=\max (e,0)$ and $e$ is non-positive on the boundary. Thus, the error remains less than zero everywhere, implying that $\Psi \le \Psi^*$, and that $\Psi$ is indeed a lower bound.

The weak maximum principle for parabolic operators can similarly be used in the case where the desirability PDE is parabolic. The only difference to note is that the augmented operator is now $P(e):= L(e) + \partial_t - \frac{1}{\lambda}q e$. For the weak maximum principle to be used, it is required that $P(e)$ have the same form but with a negative temporal derivative $P(e) = L(e) - \partial_t - \frac{1}{\lambda} q e$. This is in fact the form of our operator, as the boundary condition along the time axis is assigned only at the terminal time, and the direction of time must be flipped in the proof relative to the time of the system's evolution.
\end{proof}
\begin{rem}
This construction may be repeated for each of the objective functions found in Table \ref{tab:PDE-types},
albeit the average cost constant $c$ must be determined a-priori
in the average cost case \cite{Todorov:2009wja}.
\end{rem}

Note that the principle underlying Proposition \ref{prop:sos-opt} may in fact be repeated with the inequalities reversed in optimization
\eqref{eq:sos_pde}, resulting in a superharmonic error function. The result is that this reversed optimization is shown to be an over-approximation.

\begin{thm}
The optimization problem \eqref{eq:sos_pde} where inequality constraints are reversed and then replaced with SOS membership may be solved as a semidefinite program, and furthermore produces a
upper bound $\Psi$ of $\Psi^*$ on the domain $\mathbb{S}$.
\end{thm}

With upper and lower bounds to the optimal desirability function obtained, the distance between each and the optimal $\Psi^*$ is bounded by a known value. Also, it is straightforward to relate these bounds to the value function as well.

\begin{prop}
Given an upper (lower) bound $\Psi$, to a solution $\Psi^*$ of \eqref{eq:hjb-pde}, then $V=-\lambda \log \Psi$ is a lower (upper) bound of $V^*$, the solution to \eqref{eq:hjb-pde-value}.
\end{prop}
\begin{proof}
For $\Psi \ge \Psi^*$
\begin{eqnarray*}
V & = & - \lambda \log \Psi \\
  & \le & -\lambda \log \Psi^* \\
 & = &V^*
\end{eqnarray*}
Since $\lambda$ is always positive. Similar reasoning applies to the lower bound.
\end{proof}

\begin{rem}
Due to the nature of the log transformation \eqref{eq:log-transform}, $\Psi$ is necessarily positive
on the domain $\mathbb{S}$. This may be included as an addition constraint $\Psi \ge 0$ in
\eqref{eq:sos_pde}. However, in this case the optimization for the
lower bound of $\Psi^*$ may not converge. It is possible to instead neglect this constraint and for
its inapplicability to be remembered if the approximate desirability function $\Psi$ is in fact less
than zero at any point on the domain.
\end{rem}

\subsection{Analysis}

Some preliminary analysis of this approach demonstrates several appealing qualities. The first of these is that the convergence of the algorithm is guaranteed.

\begin{prop}
There exists a constant $c$ such that the SOS optimization problem arising from (\ref{eq:sos_pde})
has a solution for all $\gamma\ge c$\end{prop}
\begin{proof}
For the PDEs in Table \ref{tab:PDE-types} that are elliptic, all problem data is polynomial and
therefore infinitely differentiable. By the elliptic regularity theorem, the solution $\Psi$ is
infinitely differentiable and therefore continuous. As this is a linear operator on a compact set,
it is continuous if and only if it is bounded. Therefore there exists some constant $c\ge\Psi$ on
the domain $\mathbb{S}$. Similarly for the parabolic case the above holds true fore each point in
time, and integration of these finite quantities over a bounded time period also produces bounded
solutions.

A constant polynomial $p(x,t)$ may be taken to be the plane with $p(x,t)=c$. As this is a polynomial
of degree zero, it is in the set of feasible solutions to (\ref{eq:sos_pde}). Since this is a convex
problem, the existence of a feasible solution $p(x,t)$ is sufficient for the algorithm to converge.
\end{proof}

Intuitively, the previous result states that there must exist constant values that upper and lower
bound the solution to the desirability, which are of course polynomial representable. Clearly such
bounds may be quite poor in practice. However, placing this problem within a hierarchy of
optimization problems with increasing polynomial degree we have the following result.

\begin{prop}
Let $\Psi_n$ be a polynomial approximation of the desirability function with maximum degre $n$.  The
hierarchy of SOS problems consisting of solutions to \eqref{eq:sos_pde} with increasing polynomial
degree produce a sequence of solutions $\left\{ \Psi_{i},\gamma_{i}\right\} _{i\in I}$ with
monotonically decreasing $\gamma_{i}$
\end{prop}

\begin{proof}
Clearly for a sufficiently high degree polynomial $\Psi$, $\Psi^{*}$ may be represented exactly if
it is polynomial itself. Further, given a solution $\Psi$ to (\ref{eq:sos_pde}), and an additional
solution $\Psi'$ of higher degree, each with solutions $\gamma,\gamma'$ respectively,
$\gamma'\le\gamma$ as $\Psi'$ may achieve error $\gamma$ by setting its additional degrees of
freedom to zero, so the solution improves monotonically.
\end{proof}

\begin{figure}
\begin{centering}
\includegraphics[scale=0.2]{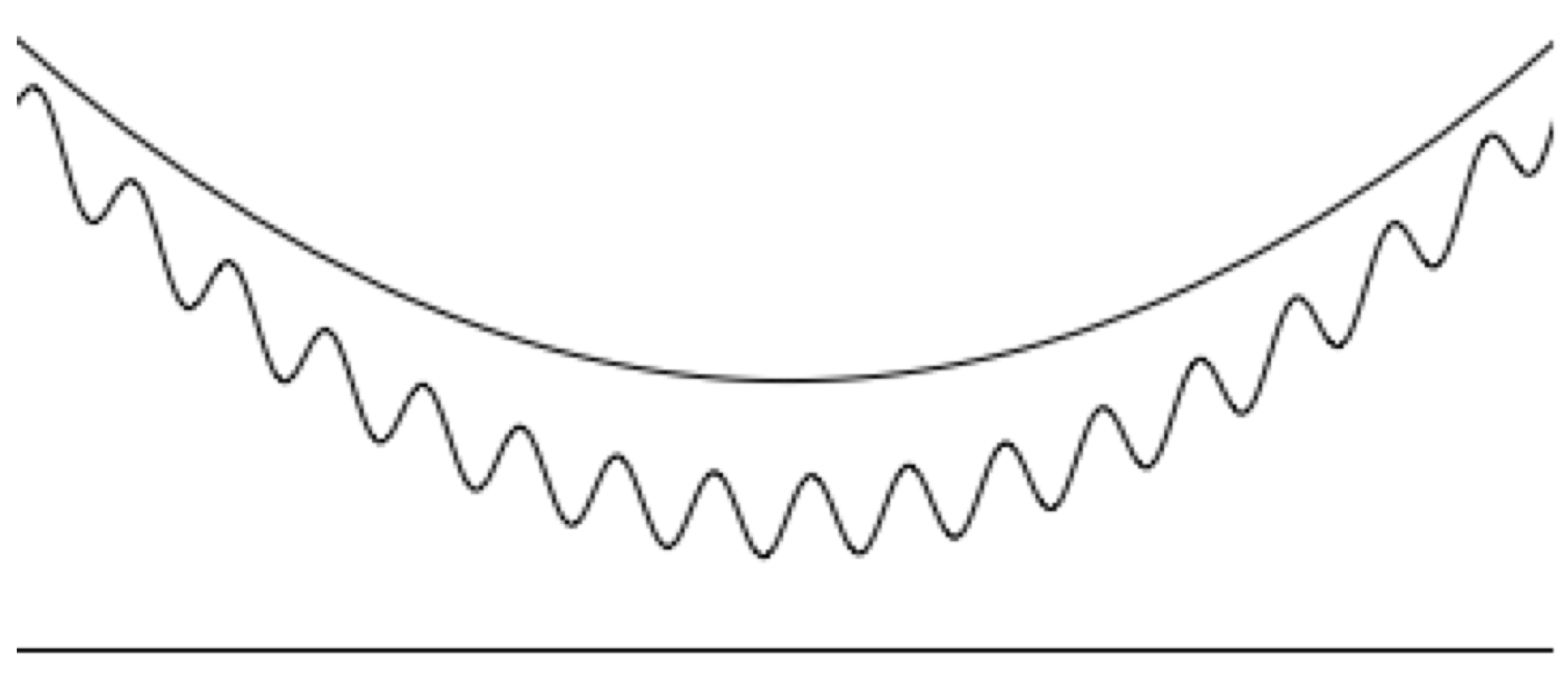}
\par\end{centering}
\caption{Illustration of potential mis-alignment between value function gradients despite proximity of approximate value function. The x-axis here denotes state space domain, while the y-axis denotes the cost-to-go at a particular state.
\label{fig:Illustration-of-potential-misalignment}}
\end{figure}

Note that we have no guarantee as to the divergence of the cost when executing the approximate value
function from the true value function.  We are only guaranteed that the value function is an
over-approximation at a particular state. A consequence is illustrated in Figure
\ref{fig:Illustration-of-potential-misalignment}.  By following the gradient, the system may diverge
significantly from the optimal path, further undermining the accuracy of the approximate value
function. This is an issue common to many approximate dynamic programming schemes
\cite{ODonoghue:2012vi, de2003linear, guestrin2003efficient}.  A common technique employed is to
simply use Monte Carlo simulation of the policy resulting from the approximate value solution,
providing an upper bound $J^{ub}$ on the realizable cost. Here, we may also flip the sign of the
inequality (\ref{eq:sos_pde}) to also obtain a lower bound. If the resulting sampled upper bound
$J^{ub}$ is near this lower bound, then the policy may be said to be empirically near-optimal.

\section{Examples}

\begin{figure}
\begin{centering}
\includegraphics[scale=0.2]{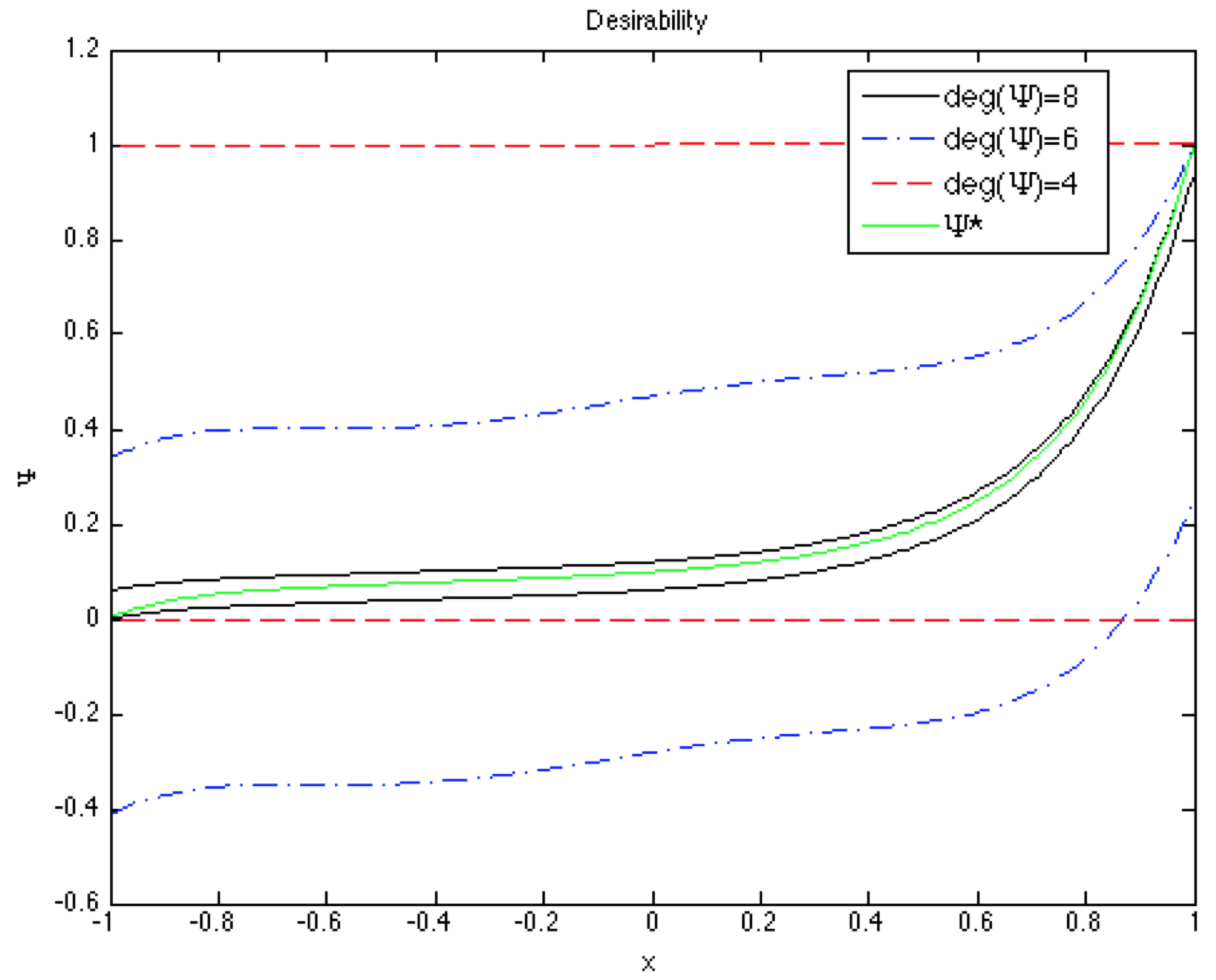}\\
\includegraphics[scale=0.2]{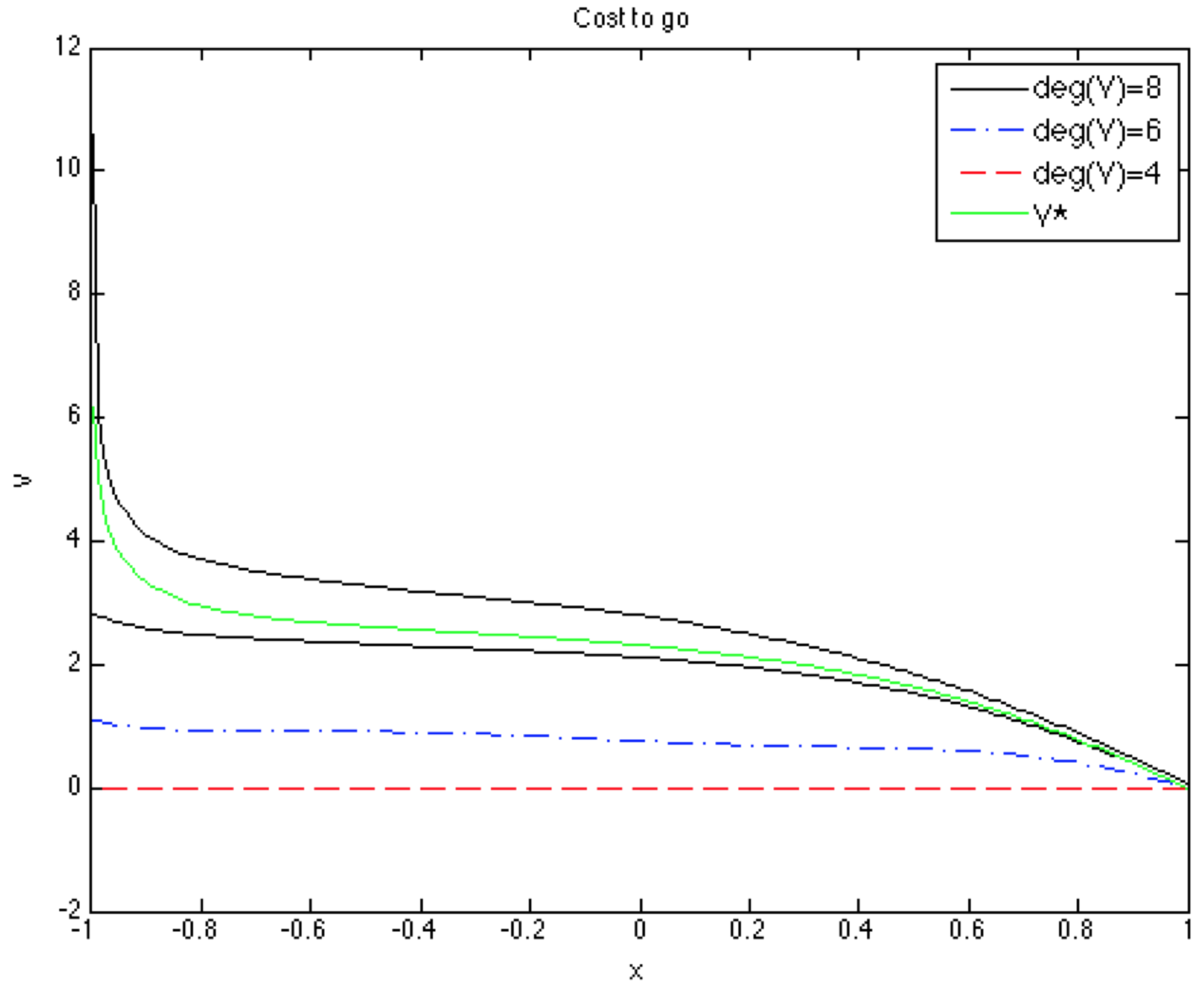}
\par\end{centering}
\caption{Plots of approximate and exact desirability and cost-to-go solutions for scalar 
system (\ref{eq:ex_1}) versus state $x$, in the interval $x\in [-1,1]$.  The dashed red,
dashed blue, and solid black lines represent the $deg(\Psi)=4$, $deg(\Psi)=6$, and $deg(\Psi)=8$
approximations. The multipliers of the Positivestellensatz were set to have matching degree, i.e. $deg(s_{i})=deg(\Psi)$.
\label{fig:Desirability-ex-1}}
\end{figure}

\begin{table} \label{table:example1}
\begin{centering}
\begin{tabular}{|c|c|c|c|c|c|}
\hline 
$deg(\Psi)$ \textbackslash{} $deg(s_{i})$ & 2 & 4 & 6 & 8 & 10\tabularnewline
\hline 
2 &1.0  &1.0  &1.0 &1.0 & 0.9994 \tabularnewline
\hline 
4 & 1.0 & 1.0 & 1.0 & 0.9999 & 0.9947 \tabularnewline
\hline 
6 &  1.0 & 1.0  & 0.7508  & 0.7498 & 0.7406 \tabularnewline
\hline 
8 &  1.0 & 1.0 & 0.2834 & 0.0592 & 0.0592 \tabularnewline
\hline
10 &  1.0 & 1.0 & 0.2834 & 0.0590 & 0.0487 \tabularnewline
\hline 
\end{tabular}
\par\end{centering}
\caption{Solution quality $\gamma$ of the desirability lower bound for varying polynomial degree of
solution $\Psi$ and Positivestellensatz multipliers $s_i$. \label{tab:ex-1-solution-table}}
\end{table}

A scalar and a two-dimensional pair of examples reveal preliminary results on the the computational
characteristics of the method. In the following problems the optimization parser Yalmip
\cite{yalmip} was used in conjunction with the semidefinite optimization package SDPT3 \cite{sdpt3}.

\subsection{Scalar System Example}

A nonlinear, unstable system with the following dynamics is considered
  \begin{equation} \label{eq:ex_1}
      dx=\left(x^{3}+5x^{2}+x+u\right)dt+d\omega
  \end{equation}
on the domain $x\in \mathbb{S} = [-1,1]$. The problem chosen is a first-exit problem, with
$\phi(-1)=10$, and $\phi(1)=0$. For this instance, $\mathcal{L}=1$, $G=1$, $B=1$, and the cost parameters
$q=1$, $R=1$ are assigned. Optimal solutions to \eqref{eq:sos_pde} of the desirability for varying
polynomial degree $deg(\Psi)$ are shown in Figure \ref{fig:Desirability-ex-1} along with its
transformed cost-to-go. The pointwise error in the desirability for increasing polynomial degree on
the solution and the multipliers is shown in Table \ref{tab:ex-1-solution-table}.  The figures
and the table clearly show that the higher the degree of polynomial approximation, the smaller the
approximation error.

\subsection{Two Dimensional Example}

Next, a nonlinear 2-dimensional problem example adapted from
\cite{Prajna:2003vr} was solved as a first-exit problem. The dynamics are set as
  \begin{eqnarray*}
   \left[\begin{array}{c}
    dx\\
    dy
   \end{array}\right] & = & \left(\left[\begin{array}{c}
                             -2x-x^{3}-5y-y^{3}\\
                              6x+x^{3}-3y-y^{3}
   \end{array}\right]+\left[\begin{array}{c}
    u_{1}\\
    u_{2}
  \end{array}\right]\right)dt\\
        & + & \left[\begin{array}{c}
     d\omega_{1}\\
     d\omega_{2}
  \end{array}\right]
\end{eqnarray*}

The system was given the task of reaching a boundary of the domain $\mathbb{S}=[-1,1]^2$, and once
there would fulfill its task with no additional cost. The control penalty was set to $R=I_{2\times 2}$, and state cost as $q(x)=0.1$. The boundary conditions for the sides
$x=-1,y=1,y=-1$ were set to have a penalty of $\phi(x,y)=1$, while for the remaining boundary $x=1$
the boundary was set to have a quadratic cost $\phi(x,y)=1-(y-1)^2$. The results are shown in Figure
\ref{fig:Results-of-multidimensional}.

\begin{figure}
\begin{centering}
\includegraphics[scale=0.2]{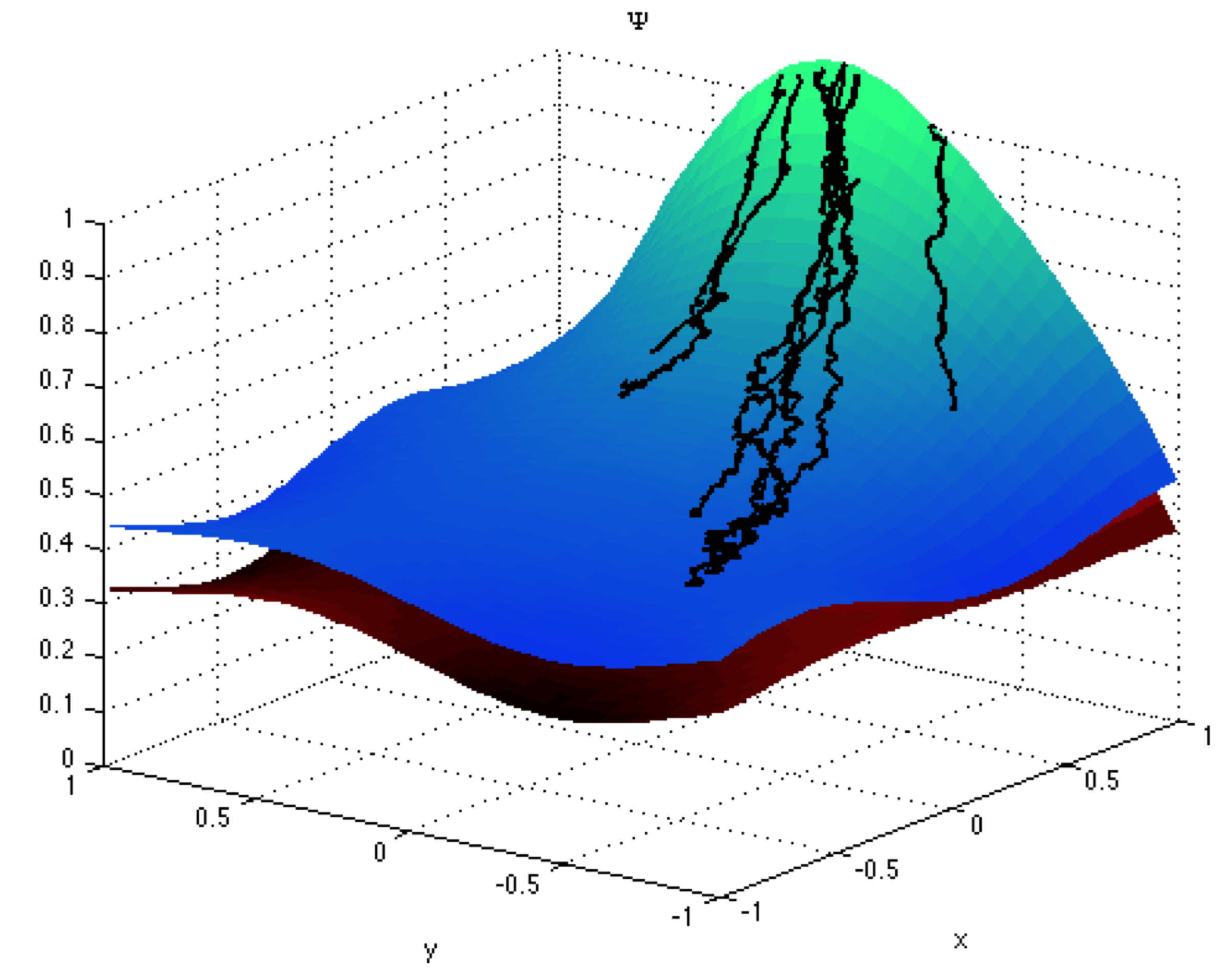}
\includegraphics[scale=0.2]{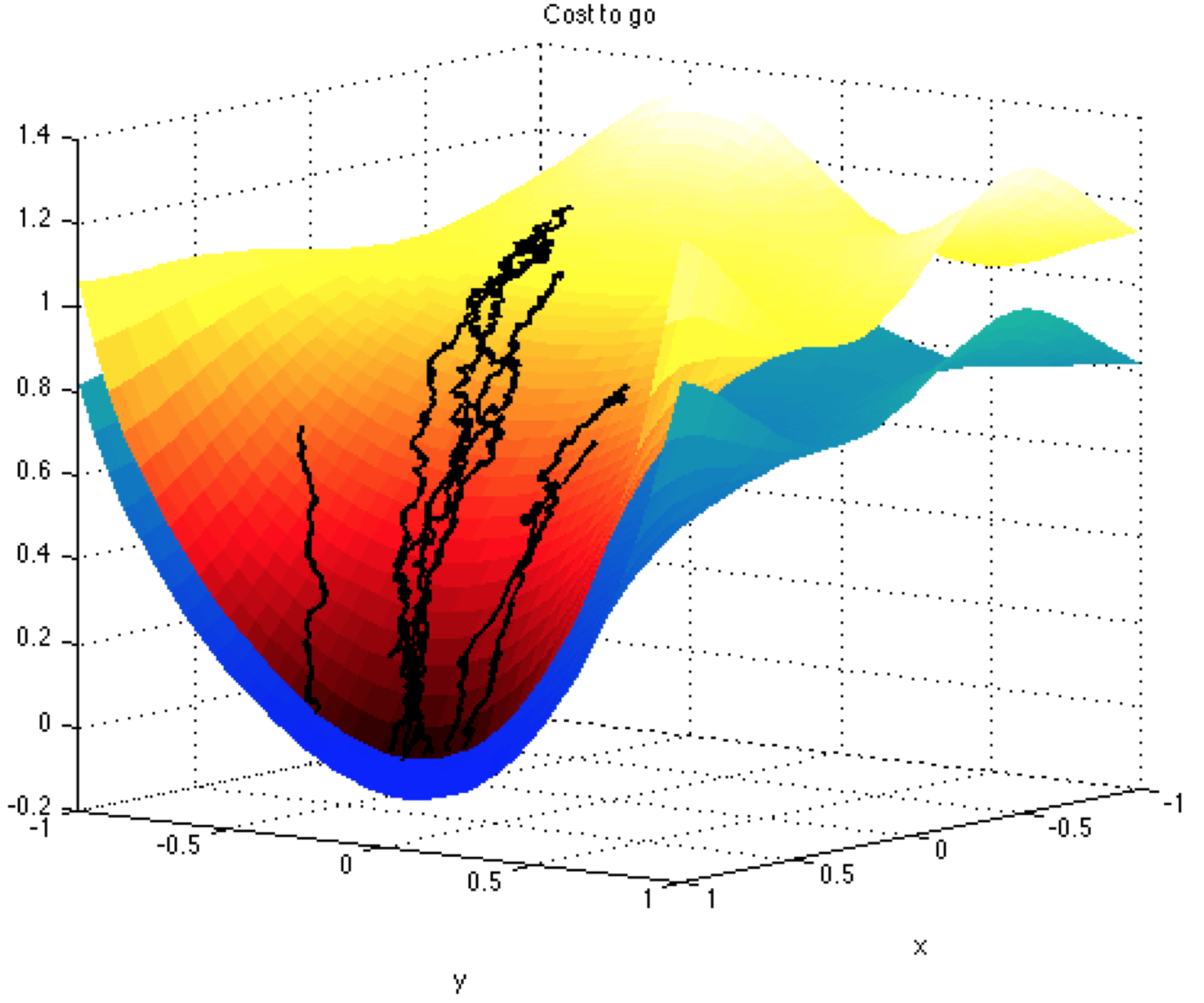}
\par\end{centering}
\caption{Desirability and Value solutions for the two dimensional example. The problem with solved
with $deg(\Psi)=deg(s_i)=14$. The upper bound had gap $\gamma_{up}=0.0979$, and lower bound
$\gamma_{lw}=0.1049$. Ten simulated trajectories of the closed loop system, randomly sampled from $x,y\in [-.75, .75]^2$ are shown in
black.\label{fig:Results-of-multidimensional}}
\end{figure}

\section{Discussion}

A method to find the value function for a class of stochastic optimal control problems was
proposed. Sum of squares and semidefinite programming was used to construct a global solution
without recourse to value iteration or other forms of dynamic programming.  The method produces
a-priori bounds on the solutions' pointwise error from the optimal HJB solution. Unfortunately,
a-priori error bounds on the cost of the trajectories resulting from policies which follow the
approximate solution were not obtained, but are the subject of further investigation. As it stands,
there is no guarantee that a specific objective will be obtained, e.g. to reach a goal region or
provide stabilization. Indeed, the mis-alignment of true and approximate value functions has
surfaced in the controls community \cite{Primbs:1999tt} as well as in the broader literature on
approximate dynamic programming \cite{de2003linear}.

The question remains of how the algorithms presented in this paper differ from the simple process of
applying approximate dynamic programming with polynomial basis functions. Key in this work is the
development in the continuous state space of the problem. Although approximate dynamic programming
aggregates states, it nonetheless begins from a discrete state space. The result is that the number
of constraints in the corresponding dynamic program depends on the size of the discrete state space
\cite{de2003linear}.  While in practice many of these constraints may be inactive, it isn't possible
to determine a-priori the inactive ones. Furthermore, as has been shown, the SOS framework gives
strong guarantees on the pointwise distance between the approximate and exact value functions.

There exists many interesting avenues for future investigation. Primary among these is the
incorporation and analysis of systems whose dynamics are not polynomial functions of state and
input. Although trigonometric functions were incorporated in several examples, a broader synthesis
that does not require ad-hoc analysis, as well as one that could incorporate discontinuities, is
needed.

Interesting connections exist with literature in the controls community.  Therein, efforts have been
made to use Lyapunov functions for optimal control, in this context dubbed Control Lyapunov
Functions. Unfortunately, methods to produce optimal control Lyapunov functions have eluded
researchers to date. The methods presented here seem promising in this light.

As mentioned, this method is proposed as an alternative to sampling based methods that utilize the
Feynman-Kac lemma. A distinct advantage of the Feynman-Kac based approach is that the required
sampling scales well with increasing dimension of the state space. It is an interesting question as
to how the method proposed here can be extended to high dimensional state spaces. The expressivity
of polynomials leads one to believe that they may be used to postpone the curse of dimensionality by
not requiring as fine a partition as a straightforward MDP-like discretization for a given desired
accuracy.

\subsection{Acknowledgements}

The authors would like to thank Venkat Chandrasakaran for guidance and suggestions.  The first author 
is grateful for the support provided by a National Science Foundation graduate fellowship.  This
work was partially supported by DARPA, through the ARM-S and DRC programs, as well as the Robotics
Technology Consortium Alliance (RCTA).

\bibliographystyle{abbrv}
\bibliography{references}

\end{document}